\documentclass{amsart}
\pdfoutput=1 
\usepackage{amssymb}
\usepackage{amsmath}
\usepackage{amsthm}
\usepackage{amsbsy}
\usepackage{bm}
\usepackage{hyperref}
\date{\today}
\usepackage{array}

\usepackage{fancyhdr}
\fancypagestyle{mypagestyle}
{
  \fancyhf{}
 \fancyhead[OC]{{ {\scriptsize NANDA KISHORE REDDY}}}
  \fancyhead[EC]{{ {\scriptsize EQUALITY OF LYAPUNOV AND STABILITY EXPONENTS}}}
\fancyhead[R]{\thepage}
}
\theoremstyle{theorem}
    \newtheorem{theorem}{Theorem}
    \newtheorem{lemma}[theorem]{Lemma}

\theoremstyle{definition} 
    \newtheorem{definition}[theorem]{Definition}
    \newtheorem{fact}[theorem]{Fact}
    
    \newtheorem{remark}[theorem]{Remark}
    \newtheorem{example}[theorem]{Example}
    \newtheorem{exercise}[theorem]{Exercise}

\def\l{\left}
\def\r{\right}
\def\<{\langle}
\def\>{\rangle}

\def\bar{\overline}
\def\P{{\bf P}}

\newcommand\tr{{\mbox{tr}}}

\newcommand\mnote[1]{} 
\newcommand\be{\begin{equation*}}

\newcommand\ee{\end{equation*}}

\newcommand\ben{\begin{equation}}
\newcommand\een{\end{equation}}
\newcommand\bes{\begin{eqnarray*}}
\newcommand\ees{\end{eqnarray*}}

\newcommand\bex{\begin{exercise}}
\newcommand\eex{\end{exercise}}
\newcommand\beg{\begin{example}}
\newcommand\eeg{\end{example}}
\newcommand\benu{\begin{enumerate}}
\newcommand\eenu{\end{enumerate}}
\newcommand\beit{\begin{itemize}}
\newcommand\eeit{\end{itemize}}
\newcommand\berk{\begin{remark}}
\newcommand\eerk{\end{remark}}
\newcommand\bdefn{\begin{defintion}}
\newcommand\edefn{\end{definition}}
\newcommand\bthm{\begin{theorem}}
\newcommand\ethm{\end{theorem}}
\newcommand\bprf{\begin{proof}}
\newcommand\eprf{\end{proof}}
\newcommand\blem{\begin{lemma}}
\newcommand\elem{\end{lemma}}

\newcommand{\sm}{{\raise0.3ex\hbox{$\scriptstyle \setminus$}}}

\def\l{\left}
\def\r{\right}

\def\CHI{\mathchoice%
{\raise2pt\hbox{$\chi$}}%
{\raise2pt\hbox{$\chi$}}%
{\raise1.3pt\hbox{$\scriptstyle\chi$}}%
{\raise0.8pt\hbox{$\scriptscriptstyle\chi$}}}
\def\smalloplus{\raise1pt\hbox{$\,\scriptstyle \oplus\;$}}

\title{\sloppy Equality of  Lyapunov and stability exponents  for products of  isotropic random matrices }

\author{Nanda Kishore Reddy}

\address{Department of Mathematics\\
        Indian Institute of Science\\
        Bangalore 560012, India}

\email{kishore11@math.iisc.ernet.in}

\thanks{Partially supported by UGC Centre for Advanced Studies. Research  supported by CSIR-SPM fellowship, CSIR, Government of India. }

\date{\today}

\begin{document}

\maketitle
\begin{abstract}

   We show that Lyapunov exponents and stability exponents are equal in the case of product of $i.i.d$ isotropic(also known as bi-unitarily invariant) random matrices. We also  derive aysmptotic distribution of singular values and eigenvalues of these product random matrices. Moreover,  Lyapunov exponents are distinct, unless the  random matrices are random scalar multiples of  Haar  unitary matrices or orthogonal matrices. As a corollary of above result, we show  probability that  product of $n$ $i.i.d$ real isotropic random matrices has all eigenvalues real goes to one as $n \to \infty$. Also,  in the proof of a lemma, we  observe that a  real (complex) Ginibre matrix can be written as product of a random lower triangular matrix and an independent truncated Haar orthogonal (unitary) matrix.
\end{abstract}

\section{Definitions and introduction}

 Let $M_1, M_2,\ldots$ be sequence of $i.i.d$  random matrices of order $d$.  Define $\bm{\sigma}_n$    to be diagonal matrix with singular values  of product matrix $\mathcal{P}_n=M_1M_2..M_n$ in the diagonal  in decreasing order and similarly  $\bm{\lambda}_n$ to be diagonal matrix with eigenvalues  of  $\mathcal{P}_n$ in the diagonal  in decreasing order of absolute values, for  $n=1,2,\ldots$. Let $|\bm{\lambda}_n|^\frac{1}{n}$ and $|\bm{\sigma}_n|^\frac{1}{n}$ denote diagonal matrices with non-negative $n$-th roots of absolute values of diagonal entries of $\bm{\lambda}_n$ and $\bm{\sigma}_n$ in the diagonal, respectively. 
 
       Define $\bm{\sigma}:= \lim_{n\to \infty} |\bm{\sigma}_n|^\frac{1}{n} $ and $\bm{\lambda}:= \lim_{n\to \infty} |\bm{\lambda}_n|^\frac{1}{n} $,  if the limits exist. Then diagonal elements of $\ln\bm{\sigma}$  and $\ln\bm{\lambda}$  are called \textit{Lyapunov exponents} and \textit{stability exponents} for products of $i.i.d$ random matrices, respectively. 
In other words, they are rates of exponential growth(or decay) of singular values and eigenvalues of   product matrices  $\mathcal{P}_n$, respectively as $n\to \infty$.   
 
 \vspace{.2cm}  
 We consider both real and complex random matrices in this paper.For the sake of simplicity, we  restrict
ourselves mostly to complex random matrices . But all the definitions and statements, along with proofs, carry over immediately  to real case (by replacing everywhere unitary matrices by  orthogonal matrices). 

 \begin{definition}
A random matrix $M$ is said to be \textit{isotropic} if probability distribution of $UMV$ is same as that of $M$, for all unitary matrices $U,V$.
 \end{definition}

 They also go by the names of bi-unitarily   invariant and rotation invariant random matrices. It  follows from definition that distribution of $UMV$ is same as that of $M$, if $U,V$ are Haar distributed random unitary matrices independent of $M$ and each other. $M=PDQ$ be the singular value decomposition  of $M$, then $UMV=UPDQV$. Since $U,V$ are Haar unitary matrices independent of $D$, then $UP$ and $QV$ are also independent Haar unitary matrices, independent of $D$.

So, it follows that $N=UDV$, where $D$  is  diagonal matrix of singular values of rotation invariant random matrix $M$ and $U,V$ are independent Haar unitary matrices independent of $D$, has same distribution as that of $M$. So, $N$ is also isotropic and this can be taken as alternate definition of isotropic random matrices. 

\begin{definition}
A product matrix $M=UDV$ is said to be \textit{isotropic} if $D$ is random diagonal matrix with non-negative real diagonal entries and $U,V$ are  Haar unitary matrices independent of $D$ and each other.
\end{definition}

\begin{definition}
A product matrix $M=UDV$ is said to be \textit{right isotropic} or \textit{right rotation invariant} if $U$ is random unitary matrix, $D$ is random diagonal matrix with non-negative real diagonal entries and $V$ is Haar unitary matrix independent of $D$ and $U$. 
\end{definition}

Lyapunov exponents for products of  $i.i.d$ random matrices and criteria for their distinctness  have been discussed in \cite{prm}. Exact expression for Lyapunov exponents in case of isotropic random matrices has been derived in \cite{newman}. In the  
same paper, explicit calculations and asymptotics  have been done in case of  Gaussian real random  matrices, also called real Ginibre ensemble. Explicit calculations of Lyapunov exponents for modified Complex Ginibre ensembles have been done in \cite{forrester4}. Stability exponents have been considered first in the setting of dynamical systems in \cite{goldman}  and therein equality of Lyapunov and stability exponents has been conjectured based upon plausible arguments and numerical results. Recent comparative studies (\cite{akemann1}, \cite{ipsen}) of  Lyapunov exponents  and stability exponents in the case of Ginibre matrix ensembles  have verified the conjecture to be true  in the respective cases. And also (\cite{forrester2}) mentions this in the case of random truncated unitary matrices. For a summary of results on this topic, we refer reader to \cite{ipsen2}. In \cite{akemann1}, a proof of this conjecture in the case
of $2 \times 2$ isotropic random matrices  has been given. In section \ref{equ}, we give a proof of this conjecture for isotropic random matrices of any order.
\vspace{.2cm}
  
Another phenomenon of interest in products of random matrices is convergence of probability, that all eigenvalues of real random product matrix $M_1M_2...M_n$ are real, to one as $n \to \infty$. This phenomenon has been first studied extensively numerically in case of real Gaussian matrices in  \cite{arul} and rigorous proofs of those results have been obtained  in \cite{forrester1}. Numerical study  of the same phenomenon for  random matrices with $i.i.d$ elements from various other  distributions  uniform, Laplace and Cauchy has been done in \cite{kavita}. In section \ref{final1}, we prove the same in case of  real isotropic random matrices, as a corollary to existence and distinctness  of stability exponents, generalizing the case of real Gaussian matrices in the direction of isotropic property.\vspace{.2cm}

Asymptotic distribution of first order fluctuations of action of product random matrices is known to be Gaussian, see Chapter V of part A in \cite{prm} for related central limit theorems. But the knowledge of variances of these distributions has been very  limited until recently. In \cite{akemann1}, asymptotic distributions of first order fluctuations of (logarithm of) both singular values and moduli of eigenvalues for product of large number of Ginibre matrices have been computed and shown to be equal. In \cite{forrester2}, variances of asymptotic distributions of first order fluctuations of singular values in cases of   generalised Ginibre matrices and truncated Haar unitary matrices have been computed. In section \ref{fluc}, we compute first order asymptotic joint probability  density of singular values of  $\mathcal{P}_n$, also that of moduli of eigenvalues of $\mathcal{P}_n$ and show them to be equal. 
\vspace{.2cm}

The paper is organised as follows. In section \ref{gen} we give an algorithm for  generating products of isotropic random matrices, using singular value decompositions, which would help us to see the relation  between eigenvalues and singular values of product matrices very clearly. Thereafter it becomes very easy to deduce the equality of Lyapunov and stability exponents which is done in section \ref{equ}.  In section \ref{fluc}, we derive first order asymptotic distribution of singular values and eigenvalues of $\mathcal{P}_n$.  In final  section \ref{final1}    probability of event that all eigenvalues of product $M_1M_2..M_n$ of $n$ $i.i.d$ isotropic real random matrices   are real is shown to converge to one as $n \to \infty$.

\section{ Products of right isotropic random matrices}\label{gen}
 
  Let $\mathcal{P}_n=M_1M_2..M_n$ be product of $i.i.d$ random matrices. $\mathcal{P}_n=\mathcal{U}_n\bm{\sigma}_n\mathcal{V}_n$ be singular value decomposition of $\mathcal{P}_n$ with diagonal entries of $\bm{\sigma}_n$ in descending order. From multiplicative ergodic theorem of Oseledec \cite{oseledec}, \cite{raghu}  or Corollary 1.3 on page 79 of \cite{prm}, we know that  under certain conditions on distribution of $M_1$, $\mathcal{U}_n$ converges to a fixed unitary matrix $\mathcal{U}_{\infty}$ and $|\bm{\sigma}_n|^\frac{1}{n}$ converges to a fixed diagonal matrix $\bm{\sigma}$. But knowledge of matrices $\{\mathcal{V}_n\}_{n=1}^{\infty}$ is also required in order to study the limiting behavior of eigenvalues of $\mathcal{P}_n$. The advantage with considering isotropic or just right rotation invariant random matrices is the full knowledge of distribution of  $\{\mathcal{V}_n\}_{n=1}^{\infty}$. It is a sequence of independent Haar unitary matrices. We refer the  reader to a recent paper \cite{peter} to appreciate and understand the role of Haar unitary matrices in   random matrix theory.
\vspace{.2cm}

Let  $M_i=U_iD_iV_i$ for $i=1,2\ldots$  be sequence of $i.i.d$ right rotation invariant random matrices  where $V_1,V_2\ldots$ are all independent Haar unitary matrices, independent of positive diagonal matrices $D_1,D_2\ldots$ and unitary matrices $U_1,U_2\ldots$.
$$\mathcal{P}_{n+1}=\mathcal{P}_n{M}_{n+1}=\mathcal{U}_n\bm{\sigma}_n\mathcal{V}_nU_{n+1}D_{n+1}V_{n+1}.$$ 

Let $\bm{\sigma}_n\mathcal{V}_nU_{n+1}D_{n+1}= R_{n+1}\bm{\sigma}_{n+1}S_{n+1}$ be singular value decomposition of $\bm{\sigma}_n\mathcal{V}_nU_{n+1}D_{n+1}$ with diagonal entries of $\bm{\sigma}_{n+1}$ in descending order. Therefore 
$$\mathcal{P}_{n+1}=\mathcal{U}_{n}R_{n+1}\bm{\sigma}_{n+1}S_{n+1}{V}_{n+1}=\mathcal{U}_{n+1}
\bm{\sigma}_{n+1}\mathcal{V}_{n+1}$$ 
where $\mathcal{U}_{n+1}=\mathcal{U}_nR_{n+1}$, $\mathcal{V}_{n+1}=S_{n+1}V_{n+1}$. Since $V_{n+1}$ is Haar unitary matrix independent of preceding matrices  $\{\mathcal{U}_k, \bm{\sigma}_k, U_k, D_k\}_{k=1}^{n+1}$ and  $\{\mathcal{V}_k\}_{k=1}^n$, we have that
$\mathcal{V}_{n+1}$ is also Haar distributed and independent of  $\{\mathcal{U}_k, \bm{\sigma}_k, U_k, D_k\}_{k=1}^{n+1}$ and  $\{\mathcal{V}_k\}_{k=1}^n$. This gives us the key idea that $\mathcal{V}_{n+1}\mathcal{U}_{n+1}$ is Haar distributed and independent of $\{\mathcal{V}_k\mathcal{U}_k\}_{k=1}^{n}$.
\begin{remark}\label{remark}
For  $\{M_n\}_{n=1}^{\infty}$, a sequence  of $i.i.d$ right rotation invariant random matrices, 
$\{\mathcal{V}_n\mathcal{U}_n\}_{n=1}^{\infty}$ is a sequence of independent Haar unitary matrices.

\end{remark}

Observe that $\bm{\sigma}_{n+1}$ is diagonal matrix of singular values of 
 $\bm{\sigma}_n\mathcal{V}_nU_{n+1}D_{n+1}$. Since $\mathcal{V}_{n}$ is Haar unitary matrix independent of $\{U_{k}, D_k\}_{k=1}^{n+1}$,  $\{\mathcal{V}_k\}_{k=1}^{n-1}$ and preceding diagonal matrices  $\{ \bm{\sigma}_k\}_{k=1}^{n}$, we have that $\mathcal{V}_nU_{n+1}$ is Haar unitary independent of $\{ \bm{\sigma}_k\}_{k=1}^{n}$, $\{\mathcal{V}_kU_{k+1}\}_{k=1}^{n-1}$ and $\{D_{k}\}_{k=1}^{n+1}$. Therefore $\{\mathcal{V}_nU_{n+1}\}_{n=1}^{\infty}$ is a sequence of $i.i.d$ Haar unitary matrices generated independently of $i.i.d$ sequence of random diagonal matrices $\{D_n\}_{n=1}^{\infty}$.  The sequence of singular value diagonal matrices $\{\bm{\sigma}_n\}_{n=1}^{\infty}$ is defined recursively by setting $\bm{\sigma}_1=D_1$ and taking  $\bm{\sigma}_{n+1}$ to be the diagonal matrix with singular values of $\bm{\sigma}_n\mathcal{V}_n U_{n+1} D_{n+1}$ in the diagonal, in descending order, for $n=1,2\ldots$. 
 
 \begin{remark}\label{remark7}
The probability distribution of $\{ \bm{\sigma}_n\}_{n=1}^{\infty}$ doesn't depend on the distribution of $U_1$, left singular vectors of $M_1$. In other words, the probability distribution of $\{ \bm{\sigma}_n\}_{n=1}^{\infty}$ remains the same in the case of  $i.i.d$ isotropic random matrices whose  distribution of singular values is same as that of $M_1$. 
\end{remark}

  Using these  observations, we have an alternate algorithm of generating   products of right rotation invariant random matrices 
$\mathcal{P}_1,\mathcal{P}_2,\ldots,\mathcal{P}_n,\ldots$ sequentially.

At the first step, generate   $M_1=U_1D_1V_1$. Call $D_1$ to be $\bm{\sigma}_1$ and  $M_1$ to be $\mathcal{P}_1$.

For $n\geq 1$,  at $n+1$-step,  we have $\mathcal{P}_n=\mathcal{U}_n\bm{\sigma}_n\mathcal{V}_n$  and 

 $(1)$   generate   a diagonal matrix $D_{n+1}$ from the distribution of $D_1$ independently. Compute  singular value decomposition of $\bm{\sigma}_n \mathcal{V}_n D_{n+1}$ to get $\bm{\sigma}_n\mathcal{V}_n D_{n+1}=R_{n+1}\bm{\sigma}_{n+1}S_{n+1}$, with diagonal entries of diagonal matrix $\bm{\sigma}_{n+1}$ in descending order. Call      $\mathcal{U}_nR_{n+1}$ to be $\mathcal{U}_{n+1}$ and

 $(2)$ generate a   Haar unitary matrix $\mathcal{V}_{n+1}$ independently.  Set  $\mathcal{U}_{n+1}\bm{\sigma}_{n+1}\mathcal{V}_{n+1}$ to be $\mathcal{P}_{n+1}$.
\vspace{.2cm}

This recursive nature of singular values of  isotropic  random matrix products  has been used in \cite{kuij} to derive singular value statistics of isotropic  random matrix products with singular values of  repulsive(Vandermonde) nature. Also, in  a very recent  study \cite{mario}, an exact relation between singular value and eigenvalue statistics of isotropic random matrices with repulsive(Vandermonde) singular values and  eigenvalues has been established.
 
\section{Asymptotic relation between eigenvalues and singular values}\label{equ}    
 
 Asymptotic behaviour of singular values of product random matrices  has been discussed in detail in \cite{prm}. We combine  Proposition 5.6 from Chapter III,  Theorem 1.2  and Proposition 2.5 from Chapter IV of part A in \cite{prm} to arrive at the following statement (though part A of the book deals with real random matrices mainly, all the statements are  true in complex case also, says the author in Introduction). 
\vspace{.2cm}

\begin{fact}\label{fact}
 Let $M_1,M_2\ldots$ be sequence of $i.i.d$ invertible   random matrices of order $d$ with common distribution $\mu$, such that $\mathbb{E}(\log^{+}\|M_1\|)$ is finite (spectral norm is used). If there exists a matrix $A$ such that $\|A\|^{-1}A$ is not unitary  and $UAU^{-1}$ belongs to support of measure $\mu$ for every unitary matrix $U$, then $|\bm{\sigma}_n|^\frac{1}{n}$ converges almost surely  to a deterministic diagonal matrix $\bm{\sigma}$ and non-zero diagonal elements of $\bm{\sigma}$ are finite and  distinct. Log values of diagonal elements of $\bm{\sigma}$ are called Lyapunov exponents.
\end{fact}
\vspace{.2cm}

Additionally if we have that  $\mathbb{E}(\log|\det(M_1)|)$ exists and  is finite, then as $\det(\bm{\sigma}_n)=\prod_{k=1}^n |\det(M_k)|$, by strong law of large numbers, we get that  $\log({\det|\bm{\sigma}_n}|^\frac{1}{n})$ converges to ${\mathbb{E}(\log|\det(M_1)|)}$ almost surely and all diagonal elements of $\bm{\sigma}$ are non-zero $\i.e$ Lyapunov exponents are all finite.

We  Use a simple idea of comparing coefficients of   characteristic polynomial written in terms of both singular values and eigenvalues, Horn inequalities \cite{horn} and  Fact  \ref{fact}  regarding convergence of singular values, to prove the main theorem of this paper, which states that  as $n\to \infty$, eigenvalues of random isotropic product  matrix  $\mathcal{P}_n$ grow (or decay) exponentially at the same rate as that of corresponding singular values. 

\begin{theorem}\label{theorem}
Let $M_1,M_2\ldots$ be sequence of $i.i.d$ invertible  isotropic random matrices of  order $d$, such that $\mathbb{E}(\log^{+}\|M_1\|)$ is finite. Then, both $|\bm{\sigma}_n|^\frac{1}{n}$ and $|\bm{\lambda}_n|^\frac{1}{n}$  converge almost surely to the same deterministic diagonal matrix $\bm{\sigma}$, whose non-zero diagonal elements are all distinct. Additionally  if $\mathbb{E}(\log|\det(M_1)|)$ exists and  is finite, then all diagonal elements of $\bm{\sigma}$ are non-zero.
\end{theorem}
\begin{proof}
If $Prob\l(\frac{M_1^{*}M_1}{{\|M_1\|}^2}=I\r) = 1$, then $M_1$ is random scalar multiple of a Haar unitary matrix. In that case both $|\bm{\sigma}_n|^\frac{1}{n}$ and $|\bm{\lambda}_n|^\frac{1}{n}$ are equal to $(\prod_{k=1}^n\|M_k\|)^\frac{1}{n}I$ which converges to $e^{\mathbb{E}(\log\|M_1\|)}I$ as $n\to \infty$, by strong law of large numbers. Hence the theorem is true in this case.
\vspace{.2cm}
  
  Suppose $Prob\l(\frac{M_1^{*}M_1}{{\|M_1\|}^2}=I\r) < 1$.   Since $M_1=U_1D_1V_1$, the above assumption ensures that, with positive probability,  $D_1$ is not a scalar multiple of identity matrix. So there exists a  diagonal matrix $A$ in the support of measure of $D_1$ such that $\|A\|^{-1}A$ is not unitary.  By the  definition of isotropic random matrices, $UAU^{-1}$ belongs to the support of measure of $M_1$ for every unitary matrix $U$. Therefore by Fact \ref{fact}, $|\bm{\sigma}_n|^\frac{1}{n}$ converges almost surely  to a diagonal matrix $\bm{\sigma}$ whose non-zero diagonal elements  are finite and  distinct. Lets say that the first $r$ diagonal entries of $\bm{\sigma}$  are non-zero, $1\leq r \leq d$.

Eigenvalues of $\mathcal{P}_n = \mathcal{U}_n\bm{\sigma}_n\mathcal{V}_n$ are same as that of $\mathcal{V}_n\mathcal{U}_n\bm{\sigma}_n$.  For the sake of simplicity of notation, we write $W_n$ in place of $\mathcal{V}_n\mathcal{U}_n$. By remark \ref{remark},    $\{{W}_n\}_{n=1}^{\infty}$ is a sequence of independent Haar unitary matrices.

For $J\subseteq \{1,2,...d\}$, $[M]_J$ denotes the determinant of the matrix formed from  matrix $M$ by deleting rows and columns whose indices are not in $J$. $|J|$ denotes cardinality of $J$. From here onwards, $J$ is always a subset of $\{1,2,..d\}$.
$$
\det(zI-\bm{\lambda}_n)= \det(zI-W_n\bm{\sigma}_n) 
$$
By comparing coefficients of $z^{d-i} \hspace{2mm}$  for $i\leq r$ in the above equation, we get 
\begin{equation}\label{idea}
\sum_{|J|=i} [\bm{\lambda}_n]_J= \sum_{|J|=i}[W_n\bm{\sigma}_n]_J=\sum_{|J|=i}[W_n]_J[\bm{\sigma}_n]_J.
\end{equation}
 We can see that $|[W_n]_J|\leqslant1$  and also, $|[W_n]_J|^\frac{1}{n}\to {1}$ as  $n\to {\infty}$ for every $J\subseteq \{1,2,...d\}$  almost surely,   from Lemma \ref{lemma}, stated and proved after this proof. Also, we have noted earlier in this proof, $|\bm{\sigma}_n|^\frac{1}{n}$ converges almost surely  to a diagonal matrix $\bm{\sigma}$ whose non-zero diagonal elements  are finite and  distinct, so we have ${[\bm{\sigma}]_J} < {[\bm{\sigma}]_{\{1,2..i\}}}$ for $|J|=i$,  $J\neq {\{1,2..i\}}$. Therefore   $\frac{[\bm{\sigma}_n]_J}{[\bm{\sigma}_n]_{\{1,2..i\}}}\to 0$ for any  $|J|=i$ and $J\neq {\{1,2..i\}} $ and $|[\bm{\sigma}_n]_{\{1,2..i\}}|^\frac{1}{n} \to [\bm{\sigma}]_{\{1,2..i\}} $ almost surely as $n\to{\infty}$,  giving us almost surely
\begin{equation*}
\lim_{n\to {\infty}} {|\sum_{|J|=i} [\bm{\lambda}_n]_J|}^{\frac{1}{n}}=\lim_{n\to {\infty}} {|\sum_{|J|=i}[W_n]_J[\bm{\sigma}_n]_J|}^{\frac{1}{n}}= [\bm{\sigma}]_{\{1,2..i\}}.
\end{equation*}
Notice that $$ |\sum_{|J|=i} [\bm{\lambda}_n]_J |\leqslant\ {d \choose i}|[\bm{\lambda}_n]_{\{1,2..i\}}|\leqslant \ {d \choose i}[\bm{\sigma}_n]_{\{1,2..i\}}$$ as the diagonal entries of $\bm{\lambda}_n$ are in descending order of absolute values and second inequality follows from Horn's inequalities \cite{horn}.  From the   above expressions and sandwich theorem, we have  that  $\lim_{n\to {\infty}} |[\bm{\lambda}_n]_{\{1,2..i\}}|^\frac{1}{n} = [\bm{\sigma}]_{\{1,2..i\}}$ for all $i\leq r$. For $i> r$, $\lim_{n\to {\infty}} |[\bm{\lambda}_n]_{\{1,2..i\}}|^\frac{1}{n} = [\bm{\sigma}]_{\{1,2..i\}}= 0$, follows directly from Horn's inequalities. Therefore,
$\lim_{n\to {\infty}} |[\bm{\lambda}_n]_{\{1,2..i\}}|^\frac{1}{n} = [\bm{\sigma}]_{\{1,2..i\}}$ for all $1\leq i \leq d$. It implies that both $|\bm{\sigma}_n|^\frac{1}{n}$ and $|\bm{\lambda}_n|^\frac{1}{n}$  converge  to the same diagonal matrix $\bm{\sigma}$ almost surely.
\end{proof}

\begin{remark}\label{remark1}
We can observe from  the proof of above Theorem that, for a sequence of product random matrices  $\{\mathcal{P}_n= \mathcal{U}_n\bm{\sigma}_n\mathcal{V}_n\}_{n=1}^{\infty}$ which satisfy remark \ref{remark}, existence and distinctness (of non-zero ones)  of Lyapunov exponents implies existence of stability exponents and also their equality with Lyapunov exponents.  As remark \ref{remark7} implies that Lyapunov exponents of right isotropic random matrices are equal to that of  isotropic random matrices with the same singular value distribution, the above Theorem holds for right isotropic random matrices also.
\end{remark}

Now we proceed to prove the lemma used in the above proof.

\begin{lemma}\label{lemma}
 Let $\{{W}_n\}_{n=1}^{\infty}$ be sequence of independent Haar unitary matrices (or orthogonal matrices) of order $d$. Then $|[W_n]_J|^\frac{1}{n}\to {1}$  and $|[W_n]_J|^\frac{1}{\sqrt{n}}\to {1}$ as  $n\to {\infty}$ for every $J\subseteq \{1,2,...d\}$  almost surely.  
\end{lemma}
\begin{proof}
As there are only finitely many $J\subseteq \{1,2,...d\}$, it is enough to prove for 
every particular $J\subseteq \{1,2,...d\}$ that $|[W_n]_J|^\frac{1}{n}\to {1}$  and $|[W_n]_J|^\frac{1}{\sqrt{n}}\to {1}$ as  $n\to {\infty}$   almost surely. For any particular $J$ ($|J|=i$), since  $[W_n]_J$ is  determinant of sub-block of Haar unitary matrix (or orthogonal matrix), using invariance of measure of $W_n$ under permutations of rows and columns, we can assume without of loss of generality that $J=\{1,2,..,i\}$. \

By Borel-cantelli lemma, we know that for a sequence of $i.i.d$ random variables $\{Y_n\}_{n=1}^{\infty}$, $\frac{Y_n}{n} \to 0$ almost surely if and only if $\mathbb{E}{|Y_1|}$ is finite. Therefore $\frac{\log|[W_n]_J|}{n}\to 0$ and $\frac{\log|[W_n]_J|}{\sqrt{n}}\to 0$ almost surely as $n\to \infty$ if and only if $\mathbb{E}|\log{|[W_1]_J|}|$ and $\mathbb{E}|\log{|[W_1]_J|}|^2$  are finite. Notice that $\mathbb{E}|\log{|[W_1]_J|}|=-\mathbb{E}\log{|[W_1]_J|}$. All we need to show now is that logarithm of absolute value of determinant of a truncated Haar unitary (or orthogonal) matrix has finite expected value.  This can be done by using density of eigenvalues of truncated unitary (or orthogonal) matrices, see \cite{ortho}, \cite{unit}. 

We propose another way of doing this which also sheds some new light on many other aspects of truncated Haar Unitary or orthogonal matrices. we do it for orthogonal case. And unitary case would be  a straightforward generalization of that. The idea is to write Ginibre matrix as product of a random lower triangular matrix and an independent truncated Haar orthogonal matrix. 

Let $S$ be $i \times d$ real Ginibre matrix $i.e$ probability density of $S$ is proportional to $e^{-\frac{1}{2}\tr({S^*S})}dS$, where $dS$ denotes Lebesgue measure on $S$. By QR decomposition of $S$ $i.e$ Gram-Schmidt orthogonalization of rows from top to bottom, we get $S=TO$ where T is lower triangular matrix with non-negative diagonal entries and $O$ is $i \times d$ matrix with orthonormal rows. Lebesgue measure on $S$  can be written in terms of new variables as $dS={\prod_{j=1}^{i}{T_{jj}}^{d-j}} dT d\mathcal H ({O})$, where $dT$ denotes Lebesgue measure on non-zero entries of $T$ and $d\mathcal H ({O})$ denotes probability measure of first $i$ rows of a $d \times d$ Haar orthogonal matrix (see \cite{muir} , p. 63). For a geometric derivation of  Jacobian of QR decomposition, we refer reader to \cite{keiburg} and also to appendix of \cite{adhikari2016} (though there the argument is for complex case, it easily carries over to real case).  The joint probability density of $T, O$ is ${(\prod_{j=1}^{i}{T_{jj}}^{d-j})}e^{-\frac{1}{2}\tr({T^*T})}dT d\mathcal H ({O})$, upto the normalizing constant. $T$ and $O$ are independent. Let $S= [S_1 \,  S_2]$ and $O=[O_1 \, O_2]$ with  $S_1,O_1$ being square matrices. Then $S_1=TO_1$, where $S_1$ is real Ginibre matrix,  $O_1$ is $i \times i$ subblock of a $d \times d$  Haar orthogonal matrix independent of $T$. All  Non-zero  entries of $T$ are independent random variables with non-diagonal ones distributed as  standard normal random variables and diagonal elements $T_{j,j}$  distributed  as square-root  of  $\chi^{2}_{(d-j+1)}$ Chi-square random variable for $j=1,2\ldots,i$. Notice that if $d=i$ we have the usual QR decomposition of $S_1$. Therefore 
\begin{eqnarray*}
\mathbb{E}\log{|[W_1]_J|} &=& \mathbb{E}\log{|\det(O_1)|} \\
&=& \mathbb{E}\log{|\det(S_1)|}-\mathbb{E}\log{|\det(T)|}\\
&=&  \frac{1}{2}\sum_{j=1}^{i}  \mathbb{E}\log{\chi^{2}_{(i-j+1)}} - \frac{1}{2}\sum_{j=1}^{i}  \mathbb{E}\log{\chi^{2}_{(d-j+1)}}
\end{eqnarray*}

and similarly 
\begin{eqnarray*}
\mathbb{E}(\log{|[W_1]_J|})^2 &=& \mathbb{E}(\log{|\det(O_1)|})^2 \\
&=& \mathbb{E}(\log{|\det(S_1)|}-\log{|\det(T)|})^2\\
&\leq & 2 \mathbb{E}(\log{|\det(S_1)|})^2+ 2 \mathbb{E}(\log{|\det(T)|})^2\\
&=& 2 \mathbb{E}( \frac{1}{2}\sum_{j=1}^{i}  \log{\chi^{2}_{(i-j+1)}} )^2+ 2 \mathbb{E}(\frac{1}{2}\sum_{j=1}^{i}  \log{\chi^{2}_{(d-j+1)}})^2\\
&\leq &  \frac{i}{2}\sum_{j=1}^{i}  \mathbb{E}(\log{\chi^{2}_{(i-j+1)}})^2 + \frac{i}{2}\sum_{j=1}^{i}  \mathbb{E}(\log{\chi^{2}_{(d-j+1)}})^2
\end{eqnarray*}

Since $\mathbb{E}\log{\chi^{2}_{(n)}}$ and $\mathbb{E}(\log{\chi^{2}_{(n)}})^2$ are finite for all $n=1,2\ldots$, we have that  $\mathbb{E}|\log{|[W_1]_J|}|$ and $\mathbb{E}|\log{|[W_1]_J|}|^2$ are finite, which proves the lemma.
\end{proof}   

\begin{remark}
In the above lemma, $M_1=LO_1$ where, since $\frac{\chi^2_{(n)}}{n} \to 1$ in distribution as $n\to \infty$,  $\frac{L}{\sqrt{d}} \to I$ in distribution as $d \to \infty$. So, $\sqrt{d}O_1 \to M_1$ in distribution as $d \to \infty$.
\end{remark}

\section{Fluctuations of singular values and eigenvalues}\label{fluc}

Let $\bm{\lambda}_n=diag(\lambda_{n,1},\lambda_{n,2},..\lambda_{n,d})$ and $\bm{\sigma}_n=diag(\sigma_{n,1},\sigma_{n,2},..\sigma_{n,d})$  be as defined earlier  for all $n=1,2...$ . Assume   that $\mathbb{E}(\log^{+}\|M_1\|)$ and $\mathbb{E} (\log|\det(M_1)|)$ are finite, so that Theorem \ref{theorem} holds. Since $M_1=\mathcal{P}_1=\mathcal{U}_1\bm{\sigma}_1\mathcal{V}_1$, the above assumptions imply that $\mathbb{E}(\log^{+}\sigma_{1,1})$ and $\mathbb{E} (\log\prod_{i=1}^d \sigma_{1,i})$ are finite. Observe that $-\infty<\frac{1}{d}\mathbb{E}(\log\prod_{i=1}^d {\sigma_{1,i}}) \leq \mathbb{E}(\log \sigma_{1,1}) \leq \mathbb{E}(\log^{+}\sigma_{1,1})<\infty$ and $-\infty<\mathbb{E}(\log(\prod_{i=1}^d {\sigma_{1,i}}))- (d-1)\mathbb{E}(\log \sigma_{1,1}) \leq \mathbb{E}(\log\sigma_{1,d}) \leq \mathbb{E}{\log\sigma_{1,1}} < \infty$, which means that all the singular values of $M_1$ have finite log-moment. Theorem \ref{theorem} gives us $\bm{\sigma}$ to have   all diagonal entries positive and distinct. \\
 
    Using recursive structure of singular values $\bm{\sigma}_n$ and the fact of almost sure convergence of $\bm{\sigma}_n^{\frac{1}{n}}$  to $\bm{\sigma}$, we can approximate ${\log(\bm{\sigma}_n)}$ by a sum of $n$ $i.i.d$ random variables, to which we can apply central limit theorem to derive the Gaussian nature of first order fluctuations of $\frac{\log(\bm{\sigma}_n)}{n}$. First of all, notice that  
\begin{eqnarray*}
\det(zI-\bm{\sigma}_{n}^2) = \det(zI-\mathcal{P}_{n}\mathcal{P}_{n}^{*}) &=& \det(zI-\mathcal{P}_{n-1}M_{n}M_{n}^* \mathcal{P}_{n-1}^{*}) \\ &=& \det(zI-\bm{\sigma}_{n-1}\mathcal{V}_{n-1}M_{n}M_{n}^*\mathcal{V}_{n-1}^{*}\bm{\sigma}_{n-1})
\end{eqnarray*}

By comparing coefficients of $z^{d-i} \hspace{2mm}$  for $i\leq r$ in the above equation, we get 
$$ \sum_{|J|=i} [\bm{\sigma}_{n}^2]_J= \sum_{|J|=i}[\bm{\sigma}_{n-1}\mathcal{V}_{n-1}M_{n}M_{n}^*\mathcal{V}_{n-1}^{*}\bm{\sigma}_{n-1}]_J=\sum_{|J|=i}[\mathcal{V}_{n-1}M_{n}M_{n}^*\mathcal{V}_{n-1}^{*}]_J[\bm{\sigma}_{n-1}^2]_J.$$

$\mathcal{V}_n$ is Haar unitary matrix independent of $\mathcal{V}_1, \mathcal{V}_2...\mathcal{V}_{n-1}$ and $M_1, M_2...M_{n-1}$. So, $\mathcal{V}_n$ is independent of  $\mathcal{P}_{n-1}$. Even though $\mathcal{V}_n$ is not independent of  $\mathcal{P}_{n}$, but being the matrix of right singular vectors of a isotropic matrix $\mathcal{P}_{n}$, it is  independent of $\mathcal{P}_n\mathcal{P}_n^*$. So, $\mathcal{V}_n$  is independent of $M_nM_n^*={\mathcal{P}_{n-1}}^{-1}\mathcal{P}_n\mathcal{P}_n^*({\mathcal{P}_{n-1}}^{-1})^{*}$ and subsequently $\mathcal{V}_{n-1}M_{n}M_{n}^*\mathcal{V}_{n-1}^{*}$ . This means that $\mathcal{V}_{n}M_{n+1}M_{n+1}^*\mathcal{V}_{n}^{*}$ is independent of $\mathcal{V}_{n-1}M_{n}M_{n}^*\mathcal{V}_{n-1}^{*}$ and the preceding matrices in the sequence.  So, $\{ \mathcal{V}_{n-1}M_{n}M_{n}^*\mathcal{V}_{n-1}^{*} \}_{n=1}^{\infty}$ for $n=1,2...$ (take $\mathcal{V}_0$ to be a Haar unitary matrix independent of all other matrices involved) forms a sequence of $i.i.d$ isotropic random matrices with distribution same as that of $\mathcal{V}_{0}M_{1}M_{1}^*\mathcal{V}_{0}^{*}$ or $\mathcal{V}_{0}D_{1}^2\mathcal{V}_{0}^{*}$ or $\mathcal{V}_{0}\bm{\sigma}_{1}^2\mathcal{V}_{0}^{*}$. For the sake of simplicity of notation, let us write $\mathcal{V}_{n-1}M_{n}M_{n}^*\mathcal{V}_{n-1}^{*}$ as $X_n$ for all $n$.

$$ \sum_{|J|=i} [\bm{\sigma}_{n}^2]_J= \sum_{|J|=i}[X_{n}]_J[\bm{\sigma}_{n-1}^2]_J.$$
 This equation gives us the required recursive structure,
$$\sum_{|J|=i} [\bm{\sigma}_{n}^2]_J= \prod_{k=1}^{n} \frac{\sum_{|J|=i}[X_{k}]_J[\bm{\sigma}_{k-1}^2]_J}{\sum_{|J|=i} [\bm{\sigma}_{k-1}^2]_J}.$$   

Notice that $k$-th term in the product is convex combination of $i$-th order principal minors of matrix $X_k$. Because of the almost sure convergence of $|\bm{\sigma}_n|^\frac{1}{n}$ to $\bm{\sigma}$ (a diagonal matrix with distinct diagonal entries), $k$-th term in the product is approximately $[X_k]_{\{1,2..i\}}$ for all large values of $k$, almost surely. Using the above equalities, positivity of principal minors of hermitian matrices and the inequality $\log(a+x)\leq \log(a)+\frac{x}{a}$ for $x,a \geq 0$, we can get lower and upper bounds for logarithm of singular values $\bm{\sigma}_n$ with both bounds being very close to same  sums of $i.i.d$ random variables for large values of $n$. We get, for all $n$
\begin{equation}\label{inequal}
\sum_{k=1}^n \log([X_{k}]_{\{1,2..i\}}) + E_{n,i}   \leq  \log(\sum_{|J|=i} [\bm{\sigma}_{n}^2]_J) 
 \leq \sum_{k=1}^n \log([X_{k}]_{\{1,2..i\}}) + E_{n,i}+F_{n,i}
 \end{equation}

 where 
\[E_{n,i}=  \log\left( \prod_{k=1}^n  \frac{[\bm{\sigma}_{k-1}^2]_{\{1,2..i\}}}{\sum_{|J|=i} [\bm{\sigma}_{k-1}^2]_J}\right), F_{n,i}= \sum_{k=1}^n   \sum_{ J\neq {\{1,2..i\}} }^{|J|=i} \frac{[X_{k}]_J[\bm{\sigma}_{k-1}^2]_J}{[X_{k}]_{\{1,2..i\}}[\bm{\sigma}_{k-1}^2]_{\{1,2..i\}}}.\]

By  using  the inequality,
$[\bm{\sigma}_n^2]_{\{1,2..i\}}   \leqslant  \sum_{|J|=i} [\bm{\sigma}_n^2]_J  \leqslant {d \choose i}[\bm{\sigma}_n^2]_{\{1,2..i\}},$ and using the  inequalities \ref{inequal} of singular values, we get that for all $i=1,2..d$
and $n=1,2,..$
\begin{equation}\label{inequal2}
\begin{array}{r@{}l}
\log(\sigma_{n,i}^2) \geq \sum_{k=1}^n \log(\frac{[X_{k}]_{\{1,2..i\}}}{[X_{k}]_{\{1,2..i-1\}}}) + E_{n,i} - E_{n,i-1}- F_{n,i-1}-\log{{d \choose i}}\\  
\log(\sigma_{n,i}^2) 
 \leq \sum_{k=1}^n \log(\frac{[X_{k}]_{\{1,2..i\}}}{[X_{k}]_{\{1,2..i-1\}}}) + E_{n,i}-E_{n,i-1}+F_{n,i}+ \log{{d \choose {i-1}}},
\end{array}
\end{equation}
(say $E_{n,0}=F_{n,0}=0$ and $[X_n]_\emptyset=1$ for all $n=1,2,..$).
\\

The similar bounds can be obtained for  moduli of eigenvalues also. We obtain them  by showing closeness between moduli of eigenvalues and singular values. From Horn's inequalities \cite{horn}, we have  $|[\bm{\lambda}_n]_{\{1,2..i\}}|\leqslant [\bm{\sigma}_n]_{\{1,2..i\}}$ for $i=1,2..d$. Using the equation \ref{idea}, we can get the following inequalities  for $i=1,2,..d$ and $n=1,2,..,$,
 \begin{equation*} 
   \log (|\sum_{|J|=i}[W_n]_J[\bm{\sigma}_n]_J|) - \log{\ {d \choose i}} \leqslant \log(|[\bm{\lambda}_n]_{\{1,2..i\}}|) \leqslant \log([\bm{\sigma}_n]_{\{1,2..i\}}). 
\end{equation*}
  Writing the difference between upper bound and lower bound as $H_{n,i}$ and using the above inequalities, we get 
 \begin{equation} \label{idea2}
   \log (\sigma_{n,i})-H_{n,i} \leqslant \log(|\lambda_{n,i}|) \leqslant \log (\sigma_{n,i})+H_{n,i-1}. 
\end{equation}
where $H_{n,0}=0$, $H_{n,i}=\log{\ {d \choose i}}-\log(|[W_n]_{\{1,2..i\}}|)-\log(|\frac{\sum_{|J|=i}[W_n]_J[\bm{\sigma}_n]_J}{[W_n]_{\{1,2..i\}}[\bm{\sigma}_n]_{\{1,2..i\}}}|)$,  for all $i=1,2..d$, $n=1,2,..,$. 
\\

Now we shall see the limiting  behavior of $E_{n,i}$, $F_{n,i}$ and $H_{n,i}$. Since $|\bm{\sigma}_n|^\frac{1}{n}$ converges almost surely to $\bm{\sigma}$, a diagonal matrix with distinct positive diagonal entries, as $n \to \infty$, it follows from basic calculus that almost surely $ \frac{[\bm{\sigma}_{n}^2]_{\{1,2..i\}}}{\sum_{|J|=i} [\bm{\sigma}_{n}^2]_J}$ converges to one with error term going down to zero exponentially fast and so  $E_{n,i}$  converges almost surely to a finite random variable $E_i$ as $n\to \infty$, for $i=1,2,..d$. 
\\

If $\limsup_{n\to \infty} [\frac{[X_{n}]_J[\bm{\sigma}_{n-1}^2]_J}{[X_{n}]_{\{1,2..i\}}[\bm{\sigma}_{n-1}^2]_{\{1,2..i\}}}]^{\frac{1}{n}} < 1$ almost surely, then ,by the root test for convergence of a series, $\sum_{k=1}^n    \frac{[X_{k}]_J[\bm{\sigma}_{k-1}^2]_J}{[X_{k}]_{\{1,2..i\}}[\bm{\sigma}_{k-1}^2]_{\{1,2..i\}}} $ converges almost surely as $n\to \infty$. We already know that $\lim_{n\to\infty} [\frac{[\bm{\sigma}_{n}^2]_J}{[\bm{\sigma}_{n}^2]_{\{1,2..i\}}}]^{\frac{1}{n}} =\frac{[\bm{\sigma}^2]_J}{[\bm{\sigma}^2]_{\{1,2..i\}}}<1$ for $J \neq {\{1,2..i\}}$. Since $\{X_n\}_{n=1}^{\infty}$ is a sequence of $i.i.d$ isotropic random matrices and  $[X_{n}]_J$ has same distribution for all $J$ such that $|J|=i$, we have by triangle inequality that   $\mathbb{E}|\log(\frac{[X_{1}]_J}{[X_{1}]_{\{1,2..i\}}})| \leq 2\mathbb{E}|\log([X_1]_{\{1,2..i\}})|\leq 2\mathbb{E}|\log([\bm{\sigma}_1]_{\{1,2..i\}})|<\infty$ (because all the singular values of $M_1$(or $\bm{\sigma}_1$) have finite log-moments). Therefore by Borel-Cantelli lemma $\frac{1}{n} \log(\frac{[X_{n}]_J}{[X_{n}]_{\{1,2..i\}}})\to 0 $ or $[\frac{[X_{n}]_J}{[X_{n}]_{\{1,2..i\}}}]^{\frac{1}{n}} \to 1$ almost surely as $n\to \infty$. So, for $J \neq {\{1,2..i\}}$, $\limsup_{n\to \infty} [\frac{[X_{n}]_J[\bm{\sigma}_{n-1}^2]_J}{[X_{n}]_{\{1,2..i\}}[\bm{\sigma}_{n-1}^2]_{\{1,2..i\}}}]^{\frac{1}{n}}= \frac{[\bm{\sigma}^2]_J}{[\bm{\sigma}^2]_{\{1,2..i\}}} < 1$ which implies $\sum_{k=1}^n    \frac{[X_{k}]_J[\bm{\sigma}_{k-1}^2]_J}{[X_{k}]_{\{1,2..i\}}[\bm{\sigma}_{k-1}^2]_{\{1,2..i\}}} $ converges almost surely as $n\to \infty$. Therefore $F_{n,i}$ converges almost surely to a finite random variable $F_i$ as $n\to \infty$  for all $i=1,2..d$.\\

It follows from Lemma \ref{lemma} that $\frac{\log(|[W_n]_{\{1,2..i\}}|)}{\sqrt{n}} \to 0$ almost surely as $n \to \infty$.
Again from Lemma \ref{lemma} we get that   $\lim_{n\to \infty} |\frac{[W_{n}]_J[\bm{\sigma}_{n}]_J}{[W_{n}]_{\{1,2..i\}}[\bm{\sigma}_{n}]_{\{1,2..i\}}}|^{\frac{1}{n}}= \frac{[\bm{\sigma}]_J}{[\bm{\sigma}]_{\{1,2..i\}}} $. So,       for $J \neq {\{1,2..i\}}$,  $\lim_{n\to \infty} \frac{[W_{n}]_J[\bm{\sigma}_{n}]_J}{[W_{n}]_{\{1,2..i\}}[\bm{\sigma}_{n}]_{\{1,2..i\}}}=0$, which implies that $\log(|\frac{\sum_{|J|=i}[W_n]_J[\bm{\sigma}_n]_J}{[W_n]_{\{1,2..i\}}[\bm{\sigma}_n]_{\{1,2..i\}}}|) \to 0$ almost surely as $n \to \infty$.  Therefore $\frac{H_{n,i}}{\sqrt{n}} \to 0$  almost surely  as $n\to \infty$  for all $i=1,2,..d$.\\

 For simple notation, denote $\frac{1}{2}( E_{n,i} - E_{n,i-1}- F_{n,i-1}-\log{{d \choose i}})$, $\frac{1}{2}(E_{n,i}-E_{n,i-1}+F_{n,i}+ \log{{d \choose {i-1}}})$  by $\underline{\bm{\epsilon}}_{n,i}$ and $\bar{\bm{\epsilon}}_{n,i}$ respectively. $\underline{\bm{\epsilon}}_n$, $\bar{\bm{\epsilon}}_n$
 be the vectors  $(\underline{\bm{\epsilon}}_{n,1},\underline{\bm{\epsilon}}_{n,2},..\underline{\bm{\epsilon}}_{n,d})$, $(\bar{\bm{\epsilon}}_{n,1},\bar{\bm{\epsilon}}_{n,2},..\bar{\bm{\epsilon}}_{n,d})$ respectively. Denote the  vectors $(H_{n,0}, H_{n,1},..H_{n,d-1})$ and $(H_{n,1}, H_{n,2},..H_{n,d})$ by $\bar{H}_n$ and $\underline{H}_n$ respectively. Observe that $\lim_{n \to \infty}{\underline{\bm{\epsilon}}_n}$ and $\lim_{n \to \infty}{\bar{\bm{\epsilon}}_n}$ are   finite  and $\lim_{n \to \infty}\frac{{\bar{H}}_{n,i}}{\sqrt{n}}=\lim_{n \to \infty}\frac{{\underline{H}}_{n,i}}{\sqrt{n}}=0$ almost surely. \\
 
 Denote $\frac{[X_{n}]_{\{1,2..i\}}}{[X_{n}]_{\{1,2..i-1\}}}$ by $L_{n,i}^2$ and  $\mathcal{L}_n$ be the vector $(L_{n,1},L_{n,2},..L_{n,d})$.  We can see that  $\{\mathcal{L}_n\}_{n=1}^{\infty}$ is a sequence of $i.i.d$ random vectors, whose distribution is same as that of  $\mathcal{L}=(L_{1},L_{2},..L_{d}):=\l(\sqrt{[V^*D^2V]_{\{1\}}},\sqrt{\frac{[V^*D^2V]_{\{1,2\}}}{[V^*D^2V]_{\{1\}}}},..\sqrt{\frac{[V^*D^2V]_{\{1,2,..d\}}}{[V^*D^2V]_{\{1,2..d-1\}}}}\r)$, where $V$ is Haar unitary matrix independent of random diagonal matrix $D$ and $D$ is distributed like diagonal matrix $D_1$ of singular values of  $M_1$. With a slight abuse of notation, $\bm{\sigma}_n$ be the vector $(\sigma_{n,1},\sigma_{n,2},..\sigma_{n,d})$ and let $\bm{\lambda}_n$ be the vector $(\lambda_{n,1},\lambda_{n,2},..\lambda_{n,d})$ and  $|\bm{\lambda}_n|$ be $(|\lambda_{n,1}|,|\lambda_{n,2}|,..|\lambda_{n,d}|)$ for all $n=1,2..$.. \\

 The inequalities \ref{inequal2} and \ref{idea2} can be written  in  vector notation as \begin{eqnarray*}
\frac{\underline{\bm{\epsilon}}_n}{2\sqrt{n}}
&\leq & \sqrt{n}\l(\frac{\log(\bm{\sigma}_{n})}{n}-\mathbb{E}\log{\mathcal{L}}\r) - \sqrt{n}\l(\frac{\sum_{k=1}^n \log(\mathcal{L}_{k})}{n}-\mathbb{E}\log{\mathcal{L}}\r) \leq   \frac{\bar{\bm{\epsilon}}_n}{2\sqrt{n}}  \\
\frac{\underline{H}_n}{\sqrt{n}} 
&\leq  & \sqrt{n}\l(\frac{\log(\bm{\sigma}_{n})}{n}-\mathbb{E}\log{\mathcal{L}}\r) - \sqrt{n}\l(\frac{\log(
|\bm{\lambda}_{n}|)}{n}-\mathbb{E}\log{\mathcal{L}}\r) \leq  \frac{\bar{H}_n}{\sqrt{n}}.  
\end{eqnarray*}

Since $\lim_{n \to \infty}\frac{{\bar{H}}_{n,i}}{\sqrt{n}}=\lim_{n \to \infty}\frac{{\underline{H}}_{n,i}}{\sqrt{n}}=0$ and  $\lim_{n \to \infty}{\underline{\bm{\epsilon}}_n}$, $\lim_{n \to \infty}{\bar{\bm{\epsilon}}_n}$ are   finite   almost surely, $\sqrt{n}\l(\frac{\log(\bm{\sigma}_{n})}{n}-\mathbb{E}\log{\mathcal{L}}\r) - \sqrt{n}\l(\frac{\sum_{k=1}^n \log(\mathcal{L}_{k})}{n}-\mathbb{E}\log{\mathcal{L}}\r) \to 0$ and $\sqrt{n}\l(\frac{\log(
|\bm{\lambda}_{n}|)}{n}-\mathbb{E}\log{\mathcal{L}}\r) - \sqrt{n}\l(\frac{\sum_{k=1}^n \log(\mathcal{L}_{k})}{n}-\mathbb{E}\log{\mathcal{L}}\r) \to 0$ almost surely as $n \to \infty$. By multivariate central limit theorem for sum of $i.i.d$ random vectors, $\sqrt{n}\l(\frac{\sum_{k=1}^n \log(\mathcal{L}_{k})}{n}-\mathbb{E}\log{\mathcal{L}}\r)$ converges in distribution to a  Gaussian random vector, whose co-variance matrix is same as that of $\log{\mathcal{L}}$.\\

To understand the random vector $\mathcal{L}$, observe that by applying Gram-Schmidt orthogonalization process to the columns of $DV$(=$M$) from left to right, we can write $DV$ as product of a unitary matrix $Q$ and a upper triangular matrix $R$ with non-negative diagonal entries and   $L_1,L_2,..L_d$ are going to be the diagonal entries(from left top to right bottom) of the upper triangular matrix $R$. Geometrically speaking, $L_1$ is  length of  the first column vector of $DV$, $L_2$ is length of the projection of second column vector onto the orthogonal complement of the space generated by the first column vector and similarly $L_i$ is length of the projection of $i$-th column vector onto the orthogonal complement of the subspace generated by the first $i-1$ column vectors of $DV$(or $M$). Since $D$ is almost surely non-singular, $L_1,L_2,..L_d$ are positive almost surely.

 For $j>i$, let $M_{j\perp(i)}$ denote the projection of   $j$-th column of $M$($=DV$) onto the orthogonal complement of subspace generated by first $i$ columns of $M$. Because of right rotation invariance of distribution of $M$, the marginal probability distribution of $M_{k\perp(i)}$  is same as that of $M_{j\perp(i)}$, for any $k>i$. From $QR$ decomposition of $M$, we can see that the length of $M_{i \perp(i-1)}$ is   $L_{i}$ and the length  of $M_{i+1\perp(i-1)}$  is $\sqrt{L_{(i+1)}^2+|R_{i,i+1}|^2}$. So, $L_{i}$  and $\sqrt{L_{(i+1)}^2+|R_{i,i+1}|^2}$ have the same probability distribution .  Therefore, $L_i$'s are stochastically decreasing in order $i.e $  $ \P(L_1>t)\geq \P(L_2>t)...\geq \P(L_d>t)$ for any real $t$. 
 
 Again by right rotation invariance, since    $R_{i,j}$ (for   $j>i$) is inner product of $j$-th column of $DV$ with $i$-th column of $Q$, all $R_{i,j}$ for $j=i+1,i+2,..d$ have the same conditional distribution given first $i$ columns of $DV$, so same marginal distributions also. $R_{1,2}$  is inner product of second column of $DV$ and unit vector along first column of $DV$, so it is  non-zero with positive probability unless $D$ is random scalar multiple of identity matrix. For $k>j\geq i$, the inner product of $M_{j\perp(i-1)}$ and $M_{k\perp(i-1)}$ is $\sum_{t=i}^{j-1}R_{t,j}R_{t,k}+ L_jR_{j,k}$.  For a fixed $i$, the innerproduct  of $M_{k\perp(i-1)}$ with unit vector along $M_{j\perp(i-1)}$ has same distribution for any  $k>j\geq i$. By taking  $j=i$, $k=i+1$ first  and then $j=i+1,k=i+2$, we get that $R_{i,i+1}$ and $\frac{R_{i,i+1}R_{i,i+2}+L_{i+1}R_{i+1,i+2}}{\sqrt{L_{(i+1)}^2+|R_{i,i+1}|^2}}$  have same distribution.  If $R_{i,i+1}$ is non-zero with positive probability, then $R_{i+1,i+2}$ is also non-zero with positive probability. Otherwise $R_{i+1,i+2}=0$ almost surely which implies that $R_{i,i+1}$ and $\frac{R_{i,i+1}R_{i,i+2}}{\sqrt{L_{(i+1)}^2+|R_{i,i+1}|^2}}$ have same distribution. But we already know that $R_{i,i+1}$ and $R_{i,i+2}$ have same distribution. So, it implies that $R_{i,i+2}$ and $\frac{R_{i,i+1}R_{i,i+2}}{\sqrt{L_{(i+1)}^2+|R_{i,i+1}|^2}}$ have same distribution, which leads to contradiction, because $\frac{R_{i,i+1}R_{i,i+2}}{\sqrt{L_{(i+1)}^2+|R_{i,i+1}|^2}}$ is strictly stochastically less than $R_{i,i+2}$ as  $L_{(i+1)}$ is positive almost surely. Therefore, if $D$ is not random scalar multiple of identity matrix, by induction $R_{i,i+1}$ is non-zero with positive probability for all $i=1,2,..d-1$. This implies that  $L_i$'s are stochastically strictly decreasing in order $i.e $  $ \P(L_1>t)\geq \P(L_2>t)...\geq \P(L_d>t)$ for any real $t$ and set of $t$'s where  strict inequalities hold between any two terms is of  non-zero measure. As logarithm is a increasing function,  this ensures distinctness and decreasing order  of Lyapunov exponents $i.e$ $\mathbb{E}\log{L_1}>\mathbb{E}\log{L_2}...>\mathbb{E}\log{L_d}$. We summarize the so far of this section  into a theorem as follows.

\begin{theorem}
Let $M_1,M_2...$ be sequence of $i.i.d$ right isotropic random matrices of order $d$, such that  all singular values of $M_1$ have finite log-moments. $M_1=QR$ be $QR$ decomposition of $M_1$ $i.e$ $Q$ is unitary matrix and $R$ is upper triangular matrix with non-negative diagonal entries $L_1,L_2,..L_d$. Let $\mathcal{L}$ be the vector $(L_{1},L_{2},..L_{d})$ and $\bm{\sigma}_{n}$, $\bm{\lambda}_{n}$ be vectors of singular values and eigenvalues (in decreasing order of their absolute values) of $M_1M_2..M_n$, respectively for all $n=1,2..$.  Then 
both 
$\sqrt{n}\l(\frac{\log(\bm{\sigma}_{n})}{n}-\mathbb{E}\log{\mathcal{L}}\r)$  and $\sqrt{n}\l(\frac{\log(
|\bm{\lambda}_{n}|)}{n}-\mathbb{E}\log{\mathcal{L}}\r)$ converge in distribution to zero-mean Gaussian random vector whose covariance matrix is the same as that of $\log{\mathcal{L}}$.

\end{theorem}

In the case of Ginibre matrices, $L_1,L_2,..L_d$ are independent. $L_i$ is $\chi^2_{d-i+1}$ random variable in case of real Ginibre matrix and $\chi^2_{2(d-i+1)}$ random variable in case of complex Ginibre martrix for all $i=1,2,..d$.   This means that in the case of products of Ginibre matrices  the first order fluctuations of ordered log-singular values(also log-eigenvalues)  are independent and Gaussian, implying the permanental nature of density of unordered log-singular values(also log-eigenvalues). This result has already been obtained in \cite{akemann1} using the  exact  densities  of singular values and eigenvalues.

In the case of truncated Haar unitary(orthogonal) matrices also, $L_1,L_2,..L_d$ are independent.  By applying $QR$ decomposition to $d \times d$ left uppermost sub-block of $m \times m$ Haar unitary(orthogonal) matrix and integrating out all the variables except  $L_1,L_2,..L_d$, we get the density of  $L_1,L_2,..L_d$.  Integrating out the auxiliary variables  here is the same as in when Schur decomposition is applied to the $d \times d$ sub-block to get the eigenvalue density of truncated Haar unitary matrices(see \cite{adhikari2016}, \cite{1}). In the case of truncated Haar unitary matrices, $L_1,L_2,..L_d$ are independent and distribution of  $L_i^2$ is $Beta(d-i+1,m-d)$ for all $i=1,2...d$. In the case of truncated Haar orthogonal matrices, $L_1,L_2,..L_d$ are independent and distribution of  $L_i^2$ is $Beta(\frac{d-i+1}{2},\frac{m-d}{2})$ for all $i=1,2...d$. This result about Lyapunov exponents of truncated Haar unitary(orthogonal)  matrices has already been obtained in \cite{forrester2} using the  exact  density of truncated Haar unitary(orthogonal)  matrices.

\section{Towards the end all eigenvalues are real}\label{final1}
Eigenvalues of a real matrix are either real or appear as complex conjugate pairs. The following theorem says, for sufficiently large $n$, with high probability all eigenvalues of $n$-th real product  matrix are  real. 
\begin{theorem}
Let $M_1,M_2\ldots$ be sequence of $i.i.d$ real invertible  isotropic random matrices of  order $d$, such that all singular values of $M_1$ have finite log-moment. Then,  Probability of the event, that all eigenvalues of product matrix $M_1M_2..M_n$ are real, goes to one as $n\to \infty$.
\end{theorem}
 \begin{proof}
 
 Since all singular values of $M_1$ have finite log-moment, both  $\mathbb{E}(\log^{+}\|M_1\|)$ and $\mathbb{E}(\log|\det(M_1)|)$ are finite. So, from Theorem \ref{theorem}, we have that $|\bm{\lambda}_n|^\frac{1}{n}$ converges almost surely to a  constant diagonal matrix $\bm{\sigma}$ with non-zero distinct diagonal entries.  It implies that moduli of  eigenvalues of product matrix $\mathcal{P}_n$ grow(or decay) exponentially at distinct rates as $n\to \infty$. For  sufficiently large $n$ the moduli of eigenvalues are distinct which means no two eigenvalues are complex conjugate of each other. Let
 $\bm{\lambda}_n= diag({\lambda}_{n,1},{\lambda}_{n,2}\ldots {\lambda}_{n,d})$  and 
$\bm{\sigma}= diag({\sigma}_{1},{\sigma}_{2}\ldots {\sigma}_{d})$. We have $|{\lambda}_{n,i}|^\frac{1}{n}\to \sigma_i$ for $1\leq i \leq d$ almost surely. Almost sure convergence implies convergence in probability. Therefore

$Prob\l(||{\lambda}_{n,i}|^\frac{1}{n}- \sigma_i|< \epsilon \; \forall \; 1\leq i \leq d \r)\to 1$ as $n\to \infty$ for any $\epsilon >0$. \\* Since $\sigma_{i}$'s are distinct, it is possible to  choose $\epsilon$ such that none of the intervals $(\sigma_i-\epsilon, \sigma_i+\epsilon)$, $1 \leq i \leq d$ intersect.  Then moduli of eigenvalues  are distinct,  
\\* $Prob\l(|{\lambda}_{n,i}|^\frac{1}{n} \in (\sigma_i-\epsilon, \sigma_i+\epsilon) \; \forall \; 1\leq i \leq d \r) \leq Prob\l(  {\lambda}_{n,1},{\lambda}_{n,2}\ldots {\lambda}_{n,d} \; are \; all \; real \r)$.  Finally \\*
$Prob\l(  {\lambda}_{n,1},{\lambda}_{n,2}\ldots {\lambda}_{n,d} \; are \; all \; real \r)\to 1 $ as $n\to \infty$.
\end{proof}

\noindent{\bf Acknowledgments:} The author is  grateful to Prof. Gernot Akemann for his questions and comments on the first draft. The author  likes to thank Prof. Peter Forrester, Prof. Gernot Akemann and  Manjunath Krishnapur for being a huge and constant source of inspiration to the author.

\bibliography{bibtex}
\bibliographystyle{amsplain}

\end{document}